%% file: sos_geom_v7.tex
\documentclass[10pt, conference]{IEEEtran}
\usepackage{times}

\usepackage[numbers]{natbib}
\usepackage{multicol}
\usepackage[bookmarks=true]{hyperref}

\usepackage{graphicx} 
\usepackage{times} 
\usepackage{amsmath} 
\usepackage{amssymb}  
\usepackage{multirow}
\usepackage{pdfpages}
\usepackage{comment}
\let\labelindent\relax
\usepackage{enumitem}
\usepackage{booktabs}
\usepackage{float}
\usepackage{caption}
\usepackage{subcaption}
\def\rot#1{\rotatebox{90}{#1}}
\input{notation.tex}

\IEEEoverridecommandlockouts
\newenvironment{ghh}{\color{black}}

\begin{document}

\title{Geometry of 3D Environments\\ and Sum of Squares Polynomials}

\author{
	Amir Ali Ahmadi$^{1}$~~~Georgina Hall$^{1}$~~~Ameesh Makadia$^{2}$~~~Vikas Sindhwani$^{3}$
	\thanks{$^{1}$Dept. of Operations Research and Financial Engineering, Princeton University, Princeton NJ, USA. Partially supported by a Google Faculty Research Award.
		{\tt\small a\char`_a\char`_a@princeton.edu, gh4@princeton.edu}}%
	\thanks{$^{2}$Google, New York, USA {\tt\small makadia@google.com}}%
	\thanks{$^{2}$Google Brain, New York, USA {\tt\small sindhwani@google.com}}%
}

\newenvironment{gh}{\color{black}}

\makeatletter
\let\@oldmaketitle\@maketitle
\renewcommand{\@maketitle}{\@oldmaketitle
	\begin{center}
		\centering
		\includegraphics[height=4cm, width=0.3\linewidth]{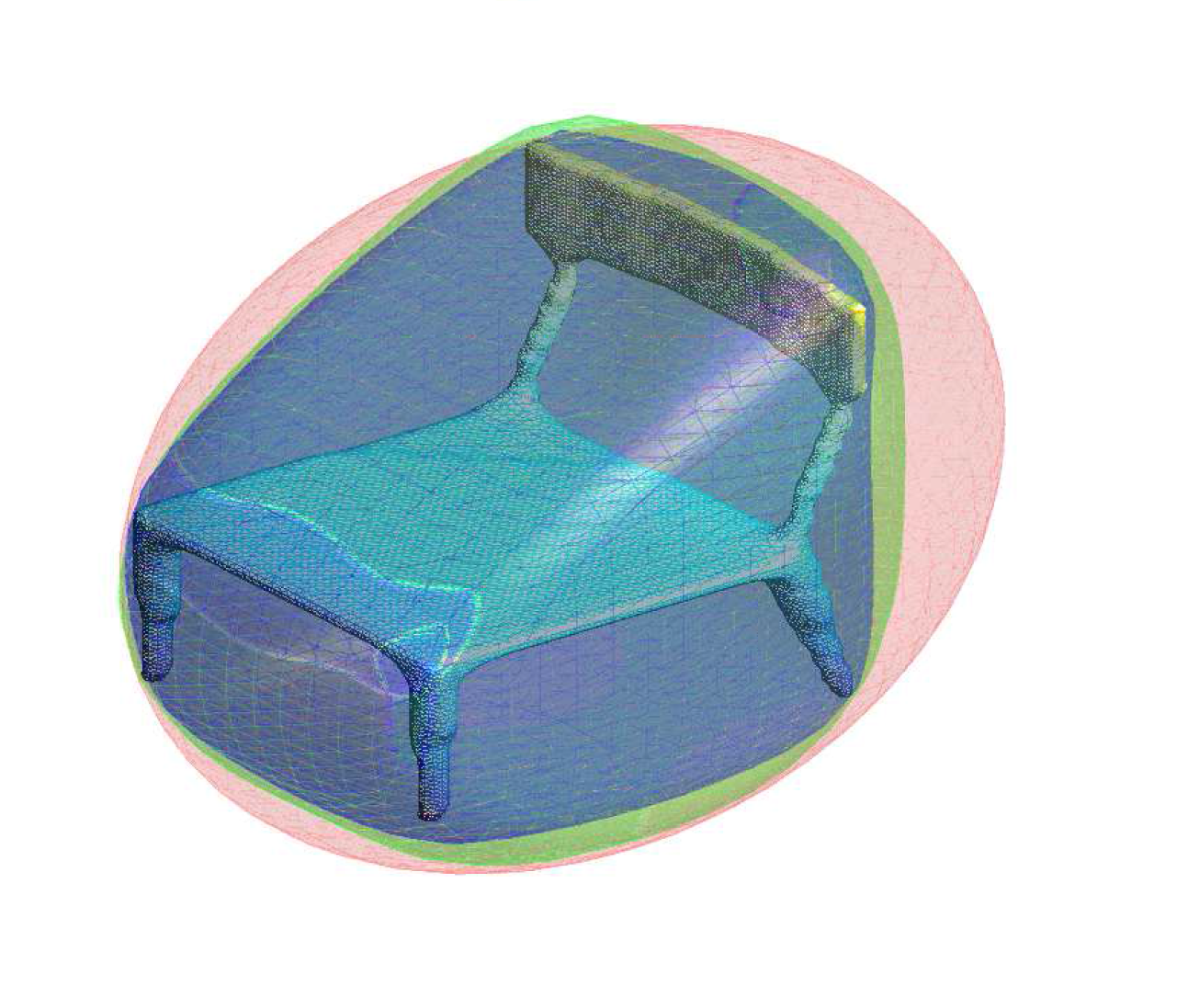}
		\includegraphics[height=4cm, width=0.3\linewidth]{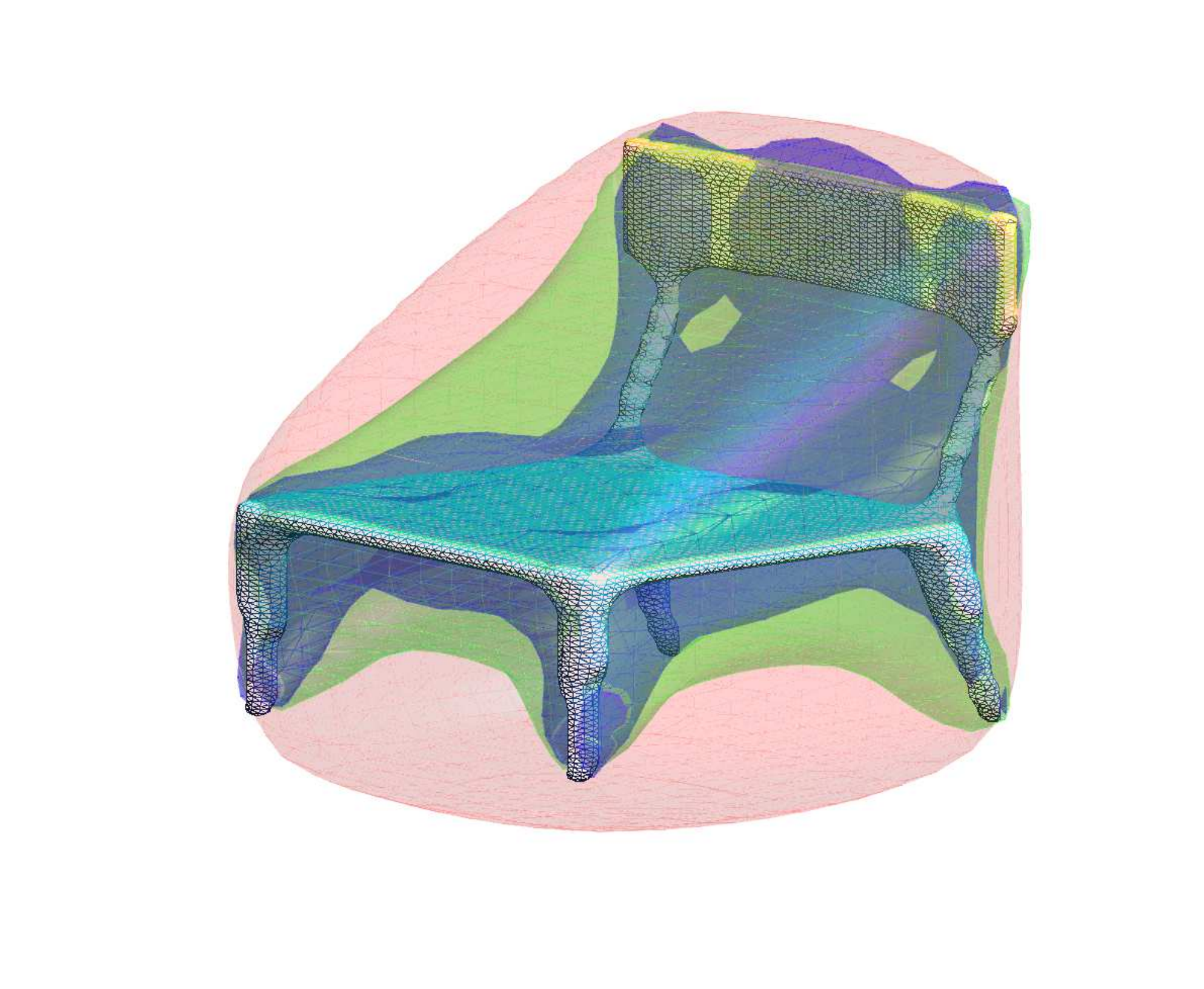}
		\includegraphics[height=4cm, width=0.3\linewidth]{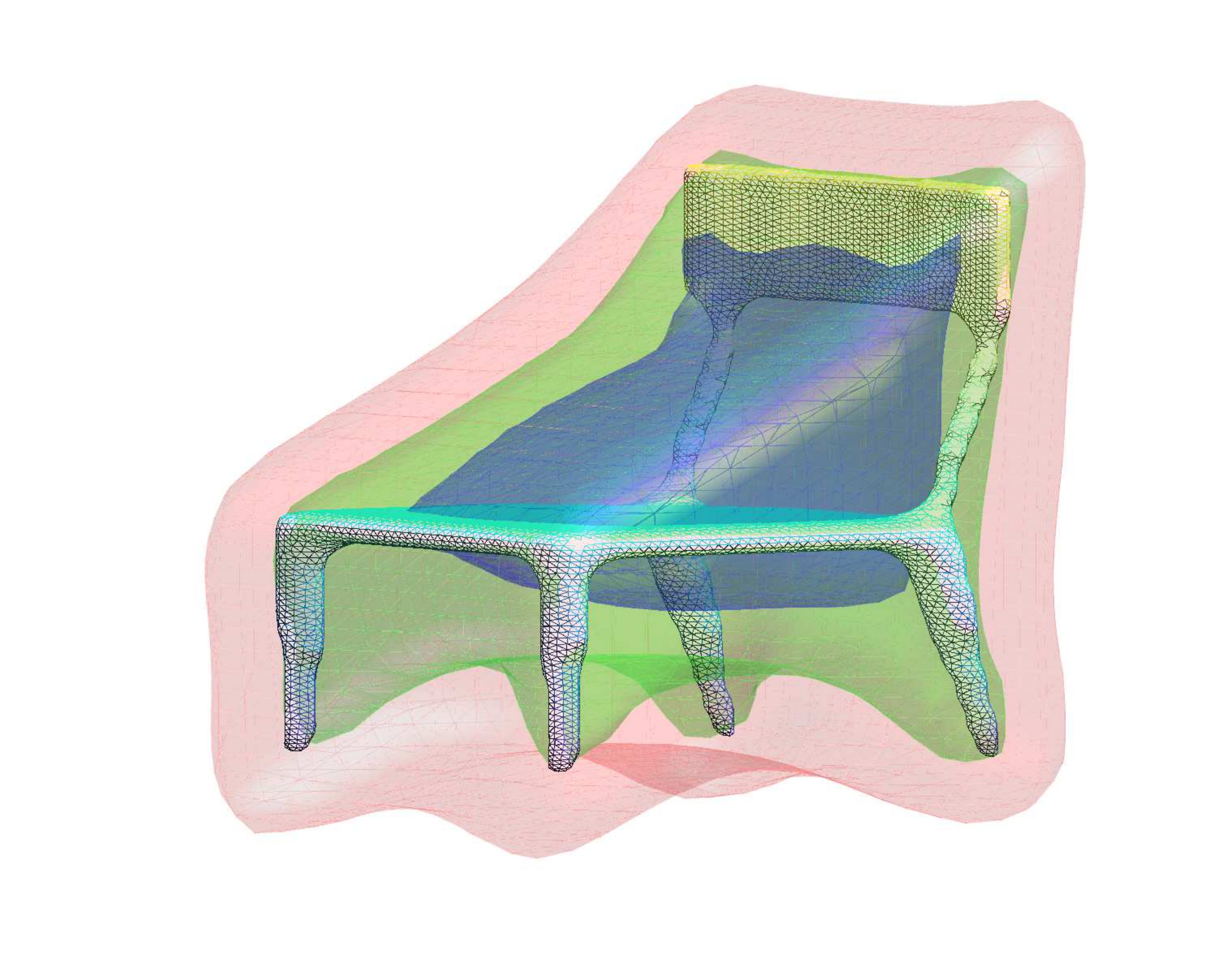}
		\captionof{figure}{Sublevel sets of sos-convex polynomials of increasing degree (left); sublevel sets of sos polynomials of increasing nonconvexity (middle); growth and shrinkage of an sos-body with sublevel sets (right)}
		\label{fig:intro_pic}
	\end{center}%
}

\makeatother

\maketitle

\begin{abstract}
Motivated by applications in robotics and computer vision, we study problems related to spatial reasoning of a 3D environment using sublevel sets of polynomials. These include: tightly containing a cloud of points (e.g., representing an obstacle) with convex or nearly-convex basic semialgebraic sets, computation of Euclidean distance between two such sets, separation of two convex basic semalgebraic sets that overlap, and tight containment of the union of several basic semialgebraic sets with a single convex one. We use algebraic techniques from sum of squares optimization that reduce all these tasks to semidefinite programs of small size and present numerical experiments in realistic scenarios. 
\end{abstract}

\IEEEpeerreviewmaketitle
\section{INTRODUCTION}
A central problem in robotics, computer graphics, virtual and augmented reality (VR/AR), and many applications involving complex physics simulations is the  accurate, real-time determination of proximity relationships between three-dimensional objects~\cite{EricsonBook} situated in a cluttered environment. In robot navigation and manipulation tasks, path planners need to compute a dynamically feasible trajectory connecting an initial state to a goal configuration while avoiding obstacles in the environment. In VR/AR applications, a human immersed in a virtual world may wish to touch computer generated objects that must respond to contacts in physically realistic ways. Likewise, when collisions are detected, 3D gaming engines and physics simulators (e.g., for molecular dynamics) need to activate appropriate directional forces on interacting entities. All of these applications require geometric notions of separation and penetration between  representations of three-dimensional objects to be continuously monitored.

A rich class of computational geometry problems arises in this context, when 3D objects are outer approximated by convex or nonconvex bounding volumes. In the case where the bounding volumes are convex, the Euclidean distance between them can be computed very precisely, providing a reliable certificate of safety for the objects they enclose. In the case where the bounding volumes are nonconvex, distance computation can be done either approximately via convex decomposition heuristics~\cite{ConvexDecomposition1,ConvexDecomposition2} which cover the volumes by a finite union of convex shapes, or exactly by using more elaborate algebraic optimization hierarchies that we discuss in this paper. When 3D objects overlap, quantitative measures of degree of penetration are needed in order to optimally resolve collisions, e.g., by a gradient-based trajectory optimizer. Multiple such measures have been proposed in the literature. The {\it penetration depth} is the minimum magnitude translation that brings the overlapping objects out of collision.  The {\it growth distance}~\cite{GrowthDistance} is the minimum shrinkage of the two bodies required to reduce volume penetration down to merely surface touching. Efficient computation of penetration measures is also a problem of interest to this paper.

\subsection{Contributions and organization of the paper}

In this work, we propose to represent the geometry of a given 3D environment comprising multiple static or dynamic rigid bodies using sublevel sets of  polynomials. The paper is organized as follows: In Section \ref{sec:sos.convex}, we provide an overview of the algebraic concepts of sum of squares (sos) and sum of squares-convex (sos-convex) polynomials as well as their relation to semidefinite programming and polynomial optimization. In Section \ref{sec:3D.point.cont}, we consider the problem of containing a cloud of 3D points with tight-fitting convex or nearly convex sublevel sets of polynomials. In particular, we propose and justify a new volume minimization heuristic for these sublevel sets which empirically results in tighter fitting polynomials than previous proposals~\cite{Magnani},~\cite{lasserre2016_inverse_moment}. Additionally, we give a procedure for explicitly tuning the extent of convexity imposed on these sublevel set bounding volumes using sum of squares optimization techniques. If convexity is imposed, we refer to them as {\it sos-convex bodies}; if it is not, we term them simply as {\it sos-bodies}. (See Section~\ref{sec:sos.convex} for a more formal definition.) We show that the bounding volumes we obtain are highly compact and adapt to the shape of the data in more flexible ways than canned convex primitives typically used in standard bounding volume hierarchies; see Table~\ref{tab:bounding.volumes}. The construction of our bounding volumes involves small-scale semidefinite programs (SDPs) that can fit, in an offline preprocessing phase, 3D meshes with tens of thousands of data points in a few seconds. In Section \ref{sec:distance}, we give sum of squares algorithms for measuring notions of separation or penetration, including  Euclidean distance and growth distance~\cite{GrowthDistance}, of two bounding volumes representing obstacles. We show that even when convexity is lacking, we can efficiently compute (often tight) lower bounds on these measures.
In Section~\ref{sec:cont.poly.sub}, we consider the problem of grouping several obstacles (i.e., bounding volumes) within one, with the idea of making a map of the 3D environment with a lower level of resolution. A semidefinite programming based algorithm for this purpose is proposed and demonstrated via an example. We end in Section~\ref{sec:conclusions} with some future directions. 

\subsection{Preview of some experiments} Figure \ref{fig:intro_pic} gives a preview of some of the methods developed in this paper using as an example a 3D chair point cloud. On the left, we enclose the chair within the 1-sublevel set of three sos-convex polynomials with increasing degree ($2$, $4$ and $6$) leading to correspondingly tighter fits. The middle plot presents the 1-sublevel set of three degree-6 sos polynomials with increasing nonconvexity showing how tighter representations can be obtained by relaxing convexity. The right plot shows the 2, 1, and 0.75 sublevel sets of a single degree-6 sos polynomial; the 1-sublevel set colored green encloses the chair, while greater or lower values of the level set define grown and shrunk versions of the object. The computation of Euclidean distances and sublevel-based measures of separation and penetration can be done in a matter of milliseconds with techniques described in this paper. 



\section{Sum of squares and sos-convexity}\label{sec:sos.convex}

In this section, we briefly review the notions of \emph{sum of squares polynomials}, \emph{sum of squares-convexity,} and \emph{polynomial optimization} which will all be central to the geometric problems we discuss later. We refer the reader to the recent monograph~\cite{lasserreBook} for a more detailed overview of the subject. 

Throughout, we will denote the set of $n \times n$ symmetric matrices by $S^{n \times n}$ and the set of degree-$2d$ polynomials with real coefficients by $\mathbb{R}_{2d}[x]$. We say that a polynomial $p(x_1,\ldots,x_n) \in \mathbb{R}_{2d}[x]$ is \emph{nonnegative} if $p(x_1,\ldots,x_n)\geq 0, \forall x\in\mathbb{R}^n$. In many applications (including polynomial optimization that we will cover later), one would like to constrain certain coefficients of a polynomial so as to make it nonnegative. Unfortunately, even testing whether a given polynomial (of degree $2d\geq 4$) is nonnegative is NP-hard. As a consequence, we would like to replace the intractable condition that $p$ be nonnegative by a sufficient condition for it that is more tractable. One such condition is for the polynomial to have a sum of squares decomposition. We say that a polynomial $p$ is a \emph{sum of squares (sos)} if there exist polynomials $q_i$ such that $p=\sum_{i} q_i^2$. From this definition, it is clear that any sos polynomial is nonnegative, though not all nonnegative polynomials are sos; see, e.g., \cite{reznick2000some},\cite{laurent2009sums} for some counterexamples. Furthermore, requiring that a polynomial $p$ be sos is a computationally tractable condition as a consequence of the following characterization: A polynomial $p$ of degree $2d$ is sos if and only if there exists a positive semidefinite matrix $Q$ such that $p(x)=z(x)^TQz(x),$ where $z(x)$ is the vector of all monomials of degree up to $d$ \cite{parrilo2000structured}. The matrix $Q$ is sometimes called the Gram matrix of the sos decomposition and is of size $\binom{n+d}{d}\times \binom{n+d}{d}$. (Throughout the paper, we let $N\mathcal{\mathop{:}}=\binom{n+d}{d}.$) The task of finding a positive semidefinite matrix $Q$ that makes the coefficients of $p$ all equal to the coefficients of $z(x)^TQz(x)$ is a semidefinite programming problem, which can be solved in polynomial time to arbitrary accuracy~\cite{vandenberghe1996semidefinite}.



The concept of sum of squares can also be used to define a sufficient condition for convexity of polynomials known as \emph{sos-convexity}. We say that a polynomial $p$ is sos-convex if the polynomial $y^T \nabla^2 p(x)y$ in $2n$ variables $x$ and $y$ is a sum of squares. Here, $\nabla^2 p(x)$ denotes the Hessian of $p$, which is a symmetric matrix with polynomial entries.
For a polynomial of degree $2d$ in $n$ variables, one can check that the dimension of the Gram matrix associated to the sos-convexity condition is $\tilde{N}\mathcal{\mathop{:}}=n \cdot \binom{n+d-1}{d-1}$. It follows from the second order characterization of convexity that any sos-convex polynomial is convex, as $y^T\nabla^2 p(x)y$ being sos implies that $\nabla^2 p(x) \succeq 0, ~\forall x.$ The converse however is not true, though convex but not sos-convex polynomials are hard to find in practice; see \cite{ahmadi2013complete}. Through its link to sum of squares, it is easy to see that testing whether a given polynomial is sos-convex is a semidefinite program. By contrast, testing whether a polynomial of degree $2d \geq 4$ is convex is NP-hard \cite{ahmadi2013np}. 


A \emph{polynomial optimization problem} is a problem of the form
\begin{align}
\min_{x \in K} p(x), \label{eq:basic.opt}
\end{align}
where the objective $p$ is a (multivariate) polynomial and the feasible set $K$ is a basic semialgebraic set; i.e., a set defined by polynomial inequalities: $$K:=\{x~|~g_i(x)\geq 0, i=1,\ldots,m\}.$$


It is straightforward to see that problem (\ref{eq:basic.opt}) can be equivalently formulated as that of finding the largest constant $\gamma$ such that  $p(x)-\gamma\geq 0,\forall x\in K.$ {\gh It is known that, under mild conditions (specifically, under the assumption that $K$ is Archimedean \cite{laurent2009sums}), the condition $p(x)-\gamma > 0, \forall x \in K$, is equivalent to the existence of sos polynomials $\sigma_i(x)$ such that $p(x)-\gamma=\sigma_0(x)+\sum_{i=1}^m \sigma_i(x) g_i(x)$. Indeed, it is at least clear that if $x \in K$, i.e., $g_i(x)\geq 0$, then $\sigma_0(x)+\sum_{i=1}^m \sigma_i(x)g_i(x) \geq 0$ which means that $p(x)-\gamma \geq 0$. The converse is less trivial and is a consequence of the Putinar Positivstellensatz \cite{putinar1993positive}.} Using this result, problem (\ref{eq:basic.opt}) can be rewritten as
\begin{align} 
&\underset{\gamma, \sigma_i}{\max}~ \gamma \nonumber \\
&\text{s.t. } p(x)-\gamma=\sigma_0+\sum_{i=1}^m \sigma_i(x)g_i(x),\label{eq:basic.opt.sos}\\
&\sigma_i \text{ sos, } i=0,\ldots,m. \nonumber
\end{align}

For any fixed upper bound on the degrees of the polynomials $\sigma_i$, this is a semidefinite programming problem which produces a lower bound on the optimal value of (\ref{eq:basic.opt}). As the degrees of $\sigma_i$ increase, these lower bounds are guaranteed to converge to the true optimal value of (\ref{eq:basic.opt}). Note that we are making \emph{no convexity assumptions} about the polynomial optimization problem and yet solving it \emph{globally} through a sequence of semidefinite programs.

{\gh \textbf{Sum of squares and polynomial optimization in robotics.} We remark that sum of squares techniques have recently found increasing applications to a whole host of problems in robotics, including constructing Lyapunov functions \cite{ahmadi2014towards}, locomotion planning \cite{kuindersma2016optimization}, design and verification of provably safe controllers \cite{majumdar2013control,majumdar2014control}, grasping and manipulation \cite{dai2015synthesis,posa2016stability, zhou2016convex}, robot-world calibration \cite{heller2014hand}, and inverse optimal control \cite{pauwels2014inverse}, among others. 
	
	{\ghh We also remark that a different use of sum of squares optimization for finding minimum bounding volumes that contain semialgebraic sets has been considered in \cite{Henrion,Henrion1} along with some interesting control applications (see Section~\ref{sec:cont.poly.sub} for a brief description).}


\section{3D point cloud containment} \label{sec:3D.point.cont}

Throughout this section, we are interested in finding a body of minimum volume, parametrized as the 1-sublevel set of a polynomial of degree $2d$, which encloses a set of given points  $\{x_1,\ldots,x_m\}$ in $\mathbb{R}^n$.

\subsection{Convex sublevel sets}\label{subsec:conv.cont}

We focus first on finding a \emph{convex} bounding volume. Convexity is a common constraint in the bounding volume literature and it makes certain tasks (e.g., distance computation among the different bodies) simpler. In order to make a set of the form $\{x\in\mathbb{R}^3| \ p(x)\leq 1\}$ convex, we will require the polynomial $p$ to be convex. (Note that this is a sufficient but not necessary condition.) Furthermore, to have a tractable formulation, we will replace the convexity condition with an sos-convexity condition as described previously. Even after these relaxations, the problem of minimizing the volume of our sublevel sets remains a difficult one. The remainder of this section discusses several heuristics for this task.

\subsubsection{The Hessian-based approach}\label{subsec:Boyd.method}

In \cite{Magnani}, Magnani et al. propose the following heuristic to minimize the volume of the 1-sublevel set of an sos-convex polynomial:
\begin{equation}
\begin{aligned}
& &&\min_{p \in \mathbb{R}_{2d}[x],H \in S^{\tilde{N} \times \tilde{N}}} -\log \det(H) \\
&\text{s.t. } &&p \text{ sos}, \\
& &&y^T \nabla^2 p(x)y=w(x,y)^THw(x,y),~H\succeq 0, \label{eq:Magnani.log.det}\\
& && p(x_i)\leq 1, i=1,\ldots,m,
\end{aligned}
\end{equation}
where $w(x,y)$ is a vector of monomials in $x$ and $y$ of degree $1$ in $y$ and $d-1$ in $x$.
This problem outputs a polynomial $p$ whose 1-sublevel set corresponds to the bounding volume that we are interested in. A few remarks on this formulation are in order:
\begin{itemize}
	\item The last constraint simply ensures that all the data points are within the 1-sublevel set of $p$ as required. 
	\item The second constraint imposes that $p$ be sos-convex. The matrix $H$ is the Gram matrix associated with the sos condition on $y^T\nabla^2 p(x)y$. 
	\item The first constraint requires that the polynomial $p$ be sos. This is a necessary condition for boundedness of (\ref{eq:Magnani.log.det}) when $p$ is parametrized with affine terms. To see this, note that for any given positive semidefinite matrix $Q$, one can always pick the coefficients of the affine terms in such a way that the constraint $p(x_i)\leq 1$ for $i=1,\ldots,m$ be trivially satisfied. Likewise one can pick the remaining coefficients of $p$ in such a way that the sos-convexity condition is satisfied.
	The restriction to sos polynomials, however, can be done without loss of generality. Indeed, suppose that the minimum volume sublevel set was given by $\{x~|~ p(x)\leq 1\}$ where $p$ is an sos-convex polynomial. As $p$ is convex and nonaffine, $\exists \gamma\geq 0$ such that $p(x)+\gamma\geq 0$ for all $x$. Define now $q(x)\mathrel{\mathop{:}}=\frac{p(x)+\gamma}{1+\gamma}.$ We have that $\{x~|~ p(x)\leq 1\}=\{ x~|~ q(x)\leq 1\}$, but here, $q$ is sos as it is sos-convex and nonnegative \cite[Lemma 8]{helton2010semidefinite}. 
\end{itemize}


{\gh The objective function of the above formulation is motivated in part by the degree $2d=2$ case. Indeed, when $2d=2$, the sublevel sets of convex polynomials are ellipsoids of the form $\{x~|~ x^TPx+b^Tx+c\leq 1\}$ and their volume is given by $\frac43 \pi \cdot \sqrt{\det(P^{-1})}$. Hence, by minimizing $-\log \det(P)$, we would exactly minimize volume. As the matrix $P$ above is none other than the Hessian of the quadratic polynomial $x^TPx+b^Tx+c$ (up to a multiplicative constant), this partly justifies the formulation given in \cite{Magnani}. Another justification for this formulation is given in \cite{Magnani} itself and relates to curvature of the polynomial $p$. Indeed, the curvature of $p$ at a point $x$ along a direction $y$ is proportional to $y^T\nabla^2 p(x)y$. By imposing that $y^T\nabla^2 p(x)y=w(x,y)^THw(x,y),$ with $H \succeq 0$, and then maximizing $\log(\det(H))$, this formulation seeks to increase the curvature of $p$ along all directions so that its 1-sublevel set can get closer to the points $x_i$. Note that curvature maximization in all directions without regards to data distribution can be counterproductive in terms of tightness of fit, particularly in regions where the data geometry is flat (an example of this is given in Figure \ref{fig:comparison.with.Boyd}).}


{\gh A related minimum volume heuristic that we will also experiment with replaces the $\log \det$ objective with a linear one. More specifically, we introduce an extra decision variable $V \in S^{\tilde{N}\times \tilde{N}}$ and minimize $\mbox{trace}(V)$ while adding an additional constraint $\begin{bmatrix} V & I \\ I & H \end{bmatrix} \succeq 0.$
Using the Schur complement, the latter constraint can be rewritten as $V\succeq H^{-1}$. As a consequence, this trace formulation minimizes the \emph{sum} of the inverse of the eigenvalues of $H$ whereas the $\log \det$ formulation described in (\ref{eq:Magnani.log.det}) minimizes the \emph{product} of the inverse of the eigenvalues. }

\subsubsection{Our approach} \label{subsec:our.approach}

We propose here an alternative heuristic for obtaining a tight-fitting convex body containing points in $\mathbb{R}^n.$ Empirically, we validate that it tends to consistently return convex bodies of smaller volume than the ones obtained with the methods described above (see Figure~\ref{fig:comparison.with.Boyd} below for an example). It also generates a relatively smaller convex optimization problem. Our formulation is as follows:
\begin{align}
&\min_{p \in \mathbb{R}_{2d}[x],P \in S^{N \times N}} -\log \det(P) \nonumber\\
&\text{s.t. } \nonumber \\
&p(x)=z(x)^TP z(x), P\succeq 0, \nonumber\\
&p \text{ sos-convex},\label{eq:VAG.log.det}\\
& p(x_i)\leq 1, i=1,\ldots,m. \nonumber
\end{align}
{\gh One can also obtain a trace formulation of this problem by replacing the $\log \det$ objective by a trace one as it was done in the previous paragraph.}

Note that the main difference between (\ref{eq:Magnani.log.det}) and (\ref{eq:VAG.log.det}) lies in the Gram matrix chosen for the objective function. In (\ref{eq:Magnani.log.det}), the Gram matrix comes from the sos-convexity constraint, whereas in (\ref{eq:VAG.log.det}), the Gram matrix is generated by the sos constraint.

In the case where the polynomial is quadratic and convex, we saw that the formulation (\ref{eq:Magnani.log.det}) is exact as it finds the minimum volume ellipsoid containing the points. It so happens that the formulation given in (\ref{eq:VAG.log.det}) is also exact in the quadratic case, and, in fact, both formulations return the same optimal ellipsoid. As a consequence, the formulation given in (\ref{eq:VAG.log.det}) can also be viewed as a natural extension of the quadratic case. 

To provide more intuition as to why this formulation performs well, we interpret the 1-sublevel set $$S\mathrel{\mathop{:}}=\{x~|~p(x)\leq 1\}$$ of $p$ as the {\gh preimage of some set whose volume is being minimized. More precisely, consider the set $$T_1=\{z(x) \in \mathbb{R}^N~|~ x \in \mathbb{R}^n\}$$ which corresponds to the image of $\mathbb{R}^n$ under the monomial map $z(x)$ and the set $$T_2=\{y \in \mathbb{R}^N ~|~ y^TPy \leq 1 \},$$ for a positive semidefinite matrix $P$ such that $p(x)=z(x)^TPz(x).$ Then, the set $S$ is simply the preimage of the intersection of $T_1$ and $T_2$ through the mapping $z$. Indeed, for any $x \in S$, we have $p(x)=z(x)^TPz(x) \leq 1$. The hope is then that by minimizing the volume of $T_2$, we will minimize volume of the intersection $T_1 \cap T_2$ and hence that of its preimage through $z$, i.e., the set $S.$ 
	
	\begin{figure}[h]
		\centering
		\includegraphics[scale=0.55]{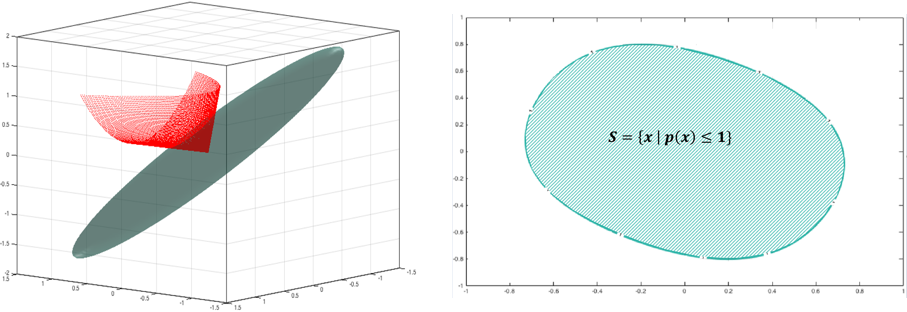}
		\caption{An illustration of the intuition behind the approach in Section \ref{subsec:our.approach}: the sets $T_1$ and $T_2$ (left) and $S$ (right)}
		\label{fig:illustration.proof}
	\end{figure}

	We illustrate this idea in Figure \ref{fig:illustration.proof}. Here, we have generated a random $3\times 3$ positive semidefinite matrix $P$ and a corresponding bivariate degree-4 sos polynomial $p(x_1,x_2)=z(x_1,x_2)^TPz(x_1,x_2)$, where $z(x_1,x_2)=(x_1^2,x_1x_2,x_2^2)^T$ is a map from $\mathbb{R}^2$ to $\mathbb{R}^3$. We have drawn in red the image of $\mathbb{R}^2$ under $z$ and in green the ellipsoid $\{y \in \mathbb{R}^3~|~y^TPy \leq 1\}.$ The preimage of the intersection of both sets seen in Figure~\ref{fig:illustration.proof} on the right corresponds to the 1-sublevel set of $p.$  }

\subsection{Relaxing convexity} \label{subsec:nonconvex}

Though containing a set of points with a convex sublevel set has its advantages, it is sometimes necessary to have a tighter fit than the one provided by a convex body, particularly if the object of interest is highly nonconvex. One way of handling such scenarios is via convex decomposition methods~\cite{ConvexDecomposition1, ConvexDecomposition2}, which would enable us to represent the object as a union of sos-convex bodies. Alternatively, one can aim for problem formulations where convexity of the sublevel sets is not imposed. In the remainder of this subsection, we first review a recent approach from the literature to do this and then present our own approach which allows for controlling the level of nonconvexity of the sublevel set. 


\subsubsection{The inverse moment approach}\label{subsubsec:Lasserre} In very recent work~\cite{lasserre2016_inverse_moment}, Lasserre and Pauwels propose an approach for containing a cloud of points with sublevel sets of polynomials (with no convexity constraint). Given a set of data points $x_1,\ldots,x_m\in\mathbb{R}^n$, it is observed in that paper that the sublevel sets of the degree $2d$ sos polynomial \begin{equation}\label{eq:inverse.moment.poly}
p_{\mu,d}(x)\mathrel{\mathop:}=z(x)^T M_d(\mu(x_1,\ldots,x_m))^{-1} z(x),
\end{equation}
tend to take the shape of the data accurately. Here, $z(x)$ is the vector of all monomials of degree up to $d$ and $M_d(\mu(x_1,\ldots,x_m))$ is the moment matrix of degree $d$ associated with the empirical measure $\mu\mathrel{\mathop:}=\frac{1}{m}\sum_{i=1}^{m}\delta_{x_i}$ defined over the data. This is an $\binom{n+d}{d} \times \binom{n+d}{d}$ symmetric positive semidefinite matrix which can be cheaply constructed from the data $x_1,\ldots,x_m\in\mathbb{R}^n$ (see~\cite{lasserre2016_inverse_moment} for details). One very nice feature of this method is that to construct the polynomial $p_{\mu, d}$ in (\ref{eq:inverse.moment.poly}) one only needs to invert a matrix (as opposed to solving a semidefinite program as our approach would require) after a single pass over the point cloud. The approach however does not a priori provide a particular sublevel set of $p_{\mu, d}$ that is guaranteed to contain all data points. Hence, once $p_{\mu, d}$ is constructed, one could slowly increase the value of a scalar $\gamma$ and check whether the $\gamma$-sublevel set of $p_{\mu, d}$ contains all points.

\subsubsection{Our approach and controlling convexity} \label{subsubsec:conv.control.ours} An advantage of our proposed formulation (\ref{eq:VAG.log.det}) is that one can easily drop the sos-convexity assumption in the constraints and thereby obtain a sublevel set which is not necessarily convex. 
This is not an option for formulation (\ref{eq:Magnani.log.det}) as the Gram matrix associated to the sos-convexity constraint intervenes in the objective.

Note that in neither this formulation nor the inverse moment approach of Lasserre and Pauwels, does the optimizer have control over the shape of the sublevel sets produced, which may be convex or far from convex. For some applications, it is useful to control in some way the degree of convexity of the sublevel sets obtained by introducing a parameter which when increased or decreased would make the sets more or less convex. This is what our following proposed optimization problem does via the parameter $c$, which corresponds in some sense to a measure of convexity:
\begin{align}
&\min_{p \in \mathbb{R}_{2d}[x],P \in S^{N \times N}} -\log \det(P) \nonumber\\
&\text{s.t. }\nonumber \\
&p=z(x)^TP z(x), P\succeq 0 \label{eq:VAG.param.convex}\\
& p(x)+ c (\sum_i x_i^2)^d \text{ sos-convex}. \nonumber \\
& p(x_i)\leq 1, i=1,\ldots,m.\nonumber
\end{align}
Note that when $c=0$, the problem we are solving corresponds exactly to (\ref{eq:VAG.log.det}) and the sublevel set obtained is convex. When $c>0$, we allow for nonconvexity of the sublevel sets. 
As we decrease $c$ towards zero, we obtain sublevel sets which get progressively more and more convex.

\subsection{Bounding volume numerical experiments} 
Figure \ref{fig:intro_pic} (left) shows the 1-sublevel sets of sos-convex bodies with degrees $2$, $4$, and $6$. A degree-$6$ polynomial gives a much tighter fit than an ellipsoid (degree 2). In the middle figure, we freeze the degree to be $6$ and increase the convexity parameter $c$ in the relaxed convexity formulation of problem~(\ref{eq:VAG.param.convex}); the 1-sublevel sets of the resulting sos polynomials with $c=0, 10, 100$ are shown. It can be seen that the sublevel sets gradually bend to better adapt to the shape of the object. The right figure shows the $2, 1,$ and $0.75$ sublevel sets of a degree-$6$ polynomial obtained by fixing $c=10$ in problem~(\ref{eq:VAG.param.convex}): the shape is retained as the body is expanded or contracted. 

{\gh Figure \ref{fig:comparison.with.Boyd} shows 1-sublevel sets of two degree-6 sos-convex polynomials. In red, we have plotted the sublevel set corresponding to maximizing curvature as explained in Section \ref{subsec:Boyd.method}. In green, we have plotted the sublevel set generated by our approach as explained in Section \ref{subsec:our.approach}. Note that our method gives a tighter-fitting sublevel set, which is in part a consequence of the flat data geometry for which the maximum curvature heuristic does not work as well.}

\begin{figure}[h]
	\centering
	\includegraphics[scale=0.3]{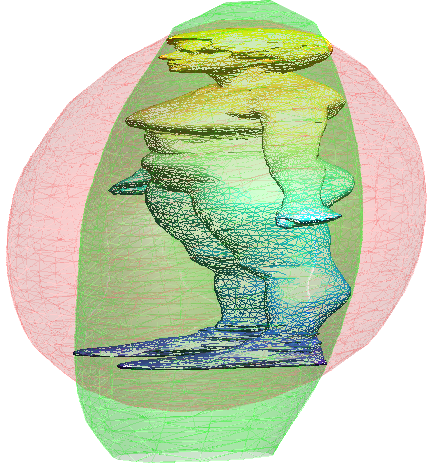}
	\caption{Comparison of degree-6 bounding volumes: our approach as described in Section \ref{subsec:our.approach} (green sublevel set) produces a tighter fitting bounding volume than the approach given in \cite{Magnani} and reviewed in Section \ref{subsec:Boyd.method} (red sublevel set). }
	\label{fig:comparison.with.Boyd}
\end{figure}

In Table \ref{tab:bounding.volumes}, we provide a comparison of various bounding volumes on Princeton Shape Benchmark datasets~\cite{ShapeData}. It can be seen that sos-convex bodies generated by higher degree polynomials provide much tighter fits than spheres or axis-aligned bounding boxes (AABB) in general. The proposed minimum volume heuristic of our formulation in (\ref{eq:VAG.log.det}) works better than that proposed in \cite{Magnani} (see (\ref{eq:Magnani.log.det})). In both formulations, typically, the log-determinant objective outperforms the trace objective. The convex hull is the tightest possible convex body. However, for smooth objects like the vase, the number of vertices describing the convex hull can be a substantial fraction of the original number of points in the point cloud. When convexity is relaxed, a degree-6 sos polynomial compactly described by just $84$ coefficients gives a tighter fit than the convex hull. For the same degree, solutions to our formulation (\ref{eq:VAG.param.convex}) with a positive value of $c$ outperform the inverse moment construction of~\cite{lasserre2016_inverse_moment}.

The bounding volume construction times are shown in Figure~\ref{fig:bv_construction_time} for sos-convex chair models. In comparison to the volume heuristics of~\cite{Magnani}, our heuristic runs noticeably faster as soon as degree exceeds $6$. {\gh We believe that this may come from the fact that the decision variable featuring in the objective in our case is a matrix of size $N \times N$, where $N=\binom{n+d}{d}$, whereas the decision variable featuring in the objective of~\cite{Magnani} is of size $\tilde{N} \times \tilde{N},$ where $\tilde{N}=n\cdot \binom{n+d-1}{d-1} > N.$}
Our implementation uses YALMIP~\cite{yalmip} with the splitting conic solver (SCS)~\cite{scs} as its backend SDP solver (run for 2500 iterations). Note that the inverse moment approach of~\cite{lasserre2016_inverse_moment} is the fastest as it does not involve any optimization and makes just one pass over the point cloud. However, this approach is not guaranteed to return a convex body, and for nonconvex bodies, tighter fitting polynomials can be estimated using log-determinant or trace objectives on our problem (\ref{eq:VAG.param.convex}).

\begin{table*}[t!]
	\begin{center}
		\begin{tabular}{|c|c|c|c|c|c|c|c|}
			\hline 
			& Object (id in \cite{ShapeData}) & & Human (10) & Chair (101) & Hand (181) & Vase (361) & Octopus (121)\\
			& $\#$ points/vertices in cvx hull&  & 9508/364 & 8499/320 & 7242/ 652 & 14859/1443 & 5944/414\\
			\hline
		 Section & Bounding Body $\downarrow$& Objective fcn $\downarrow$ & \multicolumn{5}{c|}{Volume $\downarrow$} \\
			\hline
			&Convex-Hull &  & 0.29 & 0.66 & 0.36 & 0.91 & 0.5 \\
			& Sphere & & 3.74 & 3.73 & 3.84 & 3.91 & 4.1\\
			& AABB & & 0.59 & 1.0 & 0.81 & 1.73 & 1.28\\
			\hline 
			&\multirow{2}{*}{sos-convex ($2d=2$)} &$logdet$ & 0.58 & 1.79 & 0.82 & 1.16 & 1.30\\
			& &$trace$ & 0.97 & 1.80 & 1.40 & 1.2 &1.76\\
			\cline{2-8}
			& \multirow{4}{*}{sos-convex ($2d=4$)} & $logdet(\vv{H}^{-1})$ & 0.57 & 1.55 & 0.69& 1.13 & 1.04\\
			& & $trace(\vv{H}^{-1})$ & 0.56 & 2.16 & 1.28& 1.09 &3.13 \\
			& & $logdet(\vv{P}^{-1})$ & 0.44 & 1.19 & 0.53& 1.05 &0.86\\
			\rot{\rlap{~\ref{subsec:conv.cont}}} & & $trace(\vv{P}^{-1})$ & 0.57& 1.25 & 0.92 & 1.09  &1.02\\
			\cline{2-8}
				 & \multirow{4}{*}{sos-convex ($2d=6$)} & $logdet(\vv{H}^{-1})$ & 0.57 & 1.27 & 0.58& 1.09& 0.93\\
	 & & $trace(\vv{H}^{-1})$ & 0.56 & 1.30 & 0.57 & 1.09 & 0.87\\
			& & $logdet(\vv{P}^{-1})$ & 0.41 &  1.02 & 0.45& 0.99 &0.74\\
			& & $trace(\vv{P}^{-1})$ & 0.45 & 1.21 & 0.48 & 1.03  &0.79\\
			\hline
			\rule{0pt}{8pt}
			& Inverse-Moment ($2d=2$) & & 4.02 & 1.42   & 2.14 & 1.36 &1.74\\
			& Inverse-Moment ($2d=4$) & & 1.53 & 0.95  & 0.90 & 1.25 &0.75\\
	 	\rot{\rlap{~\ref{subsubsec:Lasserre}}}	& Inverse-Moment ($2d=6$) & & 0.48 & 0.54   & 0.58 & 1.10 &0.57\\
			\hline
			& \multirow{2}{*}{sos ($2d=4, c=10$)}  
		& $logdet(\vv{P}^{-1})$ & 0.38 & 0.72 & 0.42 & 1.05 &0.63\\
			& & $trace(\vv{P}^{-1})$ & 0.51 & 0.78& 0.48 & 1.11 &0.71\\
			\cline{2-8}
		&	\multirow{2}{*}{sos ($2d=6, c=10$)}  
			& $logdet(\vv{P}^{-1})$ & 0.35 & 0.49& 0.34 &0.92& 0.41\\
		&	& $trace(\vv{P}^{-1})$ & 0.37 & 0.56 & 0.39 & 0.99&0.54\\
				\cline{2-8}
			&\multirow{2}{*}{sos ($2d=4, c=100$)}  
			& $logdet(\vv{P}^{-1})$  & 0.36 & 0.64 & 0.39 &1.05& 0.46\\
			\rot{\rlap{~\ref{subsubsec:conv.control.ours}}}  & & $trace(\vv{P}^{-1})$ & 0.42 & 0.74  & 0.46 & 1.10 &0.54\\
				\cline{2-8}
				& \multirow{2}{*}{sos ($2d=6, c=100$)}  
			& $logdet(\vv{P}^{-1})$  & 0.21 & 0.21 & 0.26 & 0.82 &0.28\\
			& & $trace(\vv{P}^{-1})$ & 0.22 & 0.30 & 0.29 & 0.85 &0.37\\
			\hline
		\end{tabular}
		\caption{ Comparison of the volume of various bounding bodies obtained from different techniques}
		\label{tab:bounding.volumes}
	\end{center}
\end{table*}

\begin{figure}[h]
	\centering
	\includegraphics[width=0.9\linewidth]{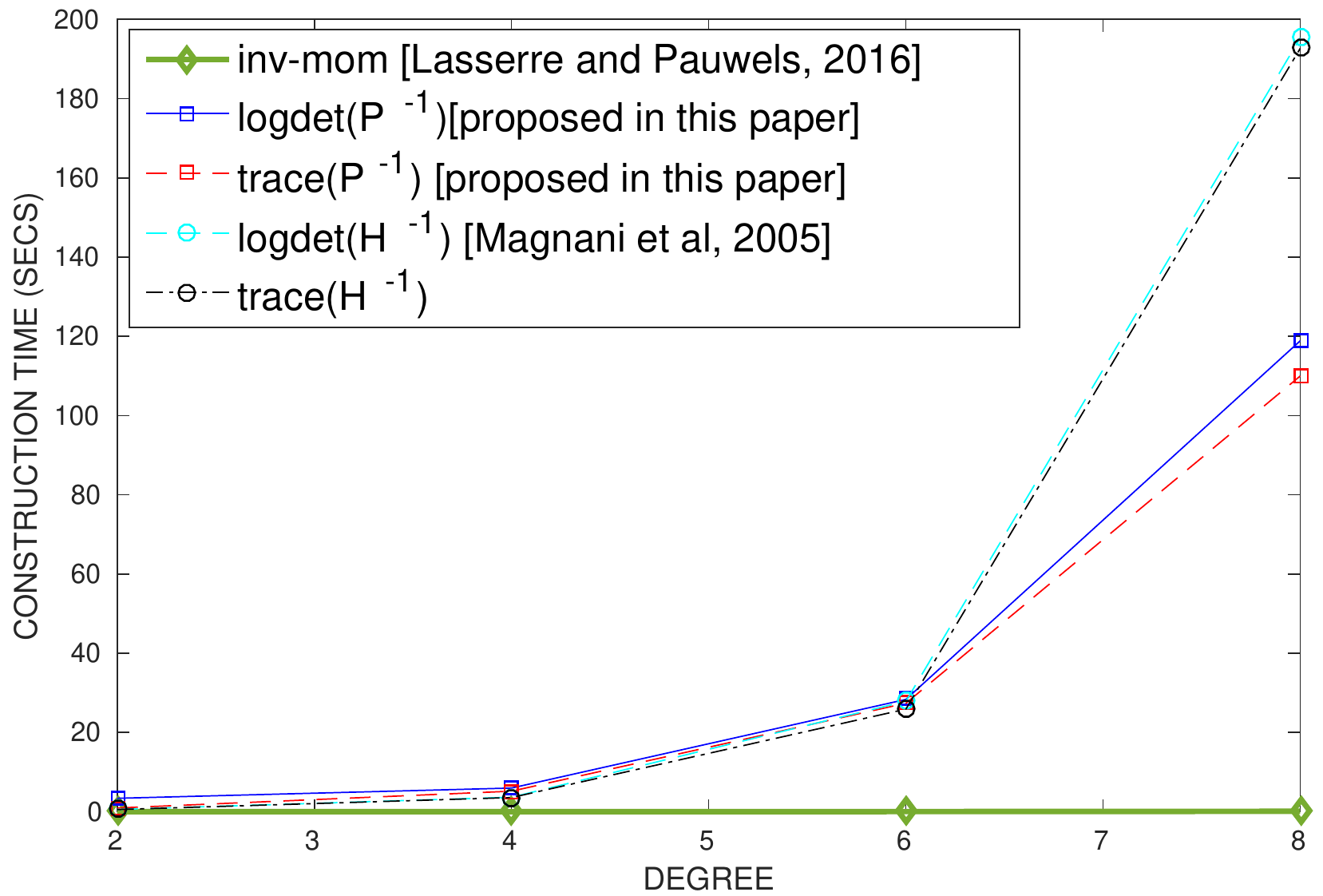}
	\caption{Bounding volume construction times}
	\label{fig:bv_construction_time}
\end{figure}

\section{Measures of Separation and Penetration}\label{sec:distance}

\subsection{Euclidean Distance} \label{subsec:eucl.dist}
{\gh In this section, we are interested in computing the Euclidean distance between two basic semialgebraic sets $$\mathcal{S}_1 \mathcal{\mathop{:}}=\{x \in \mathbb{R}^n ~|~ g_1(x)\leq 1, \ldots, g_m \leq 1\},$$ and $$\mathcal{S}_2\mathcal{\mathop{:}}=\{x \in \mathbb{R}^n ~|~ h_1(x) \leq 1, \ldots, h_r \leq 1\}$$ (where $g_1,\ldots,g_m$ and $h_1,\ldots,h_r$ are polynomials). This can be written as the following polynomial optimization problem:
	\begin{align}
	&\min_{x \in \mathcal{S}_1, y \in \mathcal{S}_2} ||x-y||_2^2. \label{eq:distance}
	\end{align}
	
	We will tackle this problem by applying the sos hierarchy described at the end of Section \ref{sec:sos.convex}. This will take the form of the following hierarchy of semidefinite programs
	\begin{equation} \label{eq:distance.lasserre}
	\begin{aligned}
	&\max_{\gamma \in \mathbb{R},\tau_i, \xi_j} \gamma\\
	&||x-y||_2^2-\gamma-\sum_{i=1}^m \tau_i(x,y) (1-g_i(x))  \\
	& \hspace{20mm} -\sum_{j=1}^r \xi_j(x,y)(1-h_j(y)) \text{ sos},\\
	&\tau_i(x,y), ~\xi_j(x,y) \text{ sos }, \forall i,\forall j,
	\end{aligned}
	\end{equation}
	where in the $d$-th level of the hierarchy, the degree of all polynomials $\tau_i$ and $\xi_j$ is upper bounded by $d$. Observe that the optimal value of each SDP produces a \emph{lower bound} on (\ref{eq:distance}) and that when $d$ increases, this lower bound can only improve.
	
	Amazingly, in all examples we tried (independently of convexity of $\mathcal{S}_1$ and $\mathcal{S}_2$!), the 0-th level of the hierarchy was already exact. By this we mean that the optimal value of (\ref{eq:distance.lasserre}) exactly matched that of (\ref{eq:distance}), already when the degree of the polynomials $\tau_i$ and $\xi_j$ was zero; i.e., when $\tau_i$ and $\xi_j$ were nonnegative scalars. An example of this phenomenon is given in Figure~\ref{fig:ex.distance.nonconvex} where the green bodies are each a (highly nonconvex) sublevel set of a quartic polynomial.
	
	When our SDP relaxation is exact, we can recover the points $x^*$ and $y^*$ where the minimum distance between sets is achieved from the eigenvector corresponding to the zero eigenvalue of the Gram matrix associated with the first sos constraint in (\ref{eq:distance.lasserre}). This is what is done in Figure~\ref{fig:ex.distance.nonconvex}.

	\begin{figure}[h] 
		\centering
		\includegraphics[scale=0.35]{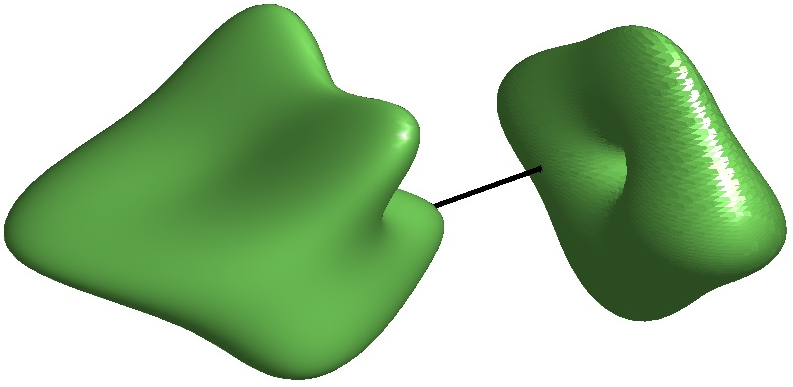}
		\caption{Minimum distance between two (nonconvex) sublevel sets of degree-4 polynomials}
		\label{fig:ex.distance.nonconvex}
	\end{figure}

	
	\textbf{The sos-convex case.} One important special case where we know that the 0-th level of the sos hierarchy in (\ref{eq:distance.lasserre}) is \emph{guaranteed} to be exact is when the defining polynomials $g_i$ and $h_i$ of $\mathcal{S}_1$ and $\mathcal{S}_2$ are \emph{sos-convex}. This is a corollary of the fact that the 0-th level sos relaxation is known to be tight for the general polynomial optimization problem in (\ref{eq:basic.opt}) if the polynomials $p$ and $-g_i$ involved in the description of $K$ there are sos-convex; see~\cite{lasserre2009convexity}. An example of the computation of the minimum distance between two degree-6 sos-convex bodies enclosing human and chair 3D point clouds is given below, together with the points achieving the minimum distance.
	\begin{figure}[h] 
		\centering
		\includegraphics[scale=0.25]{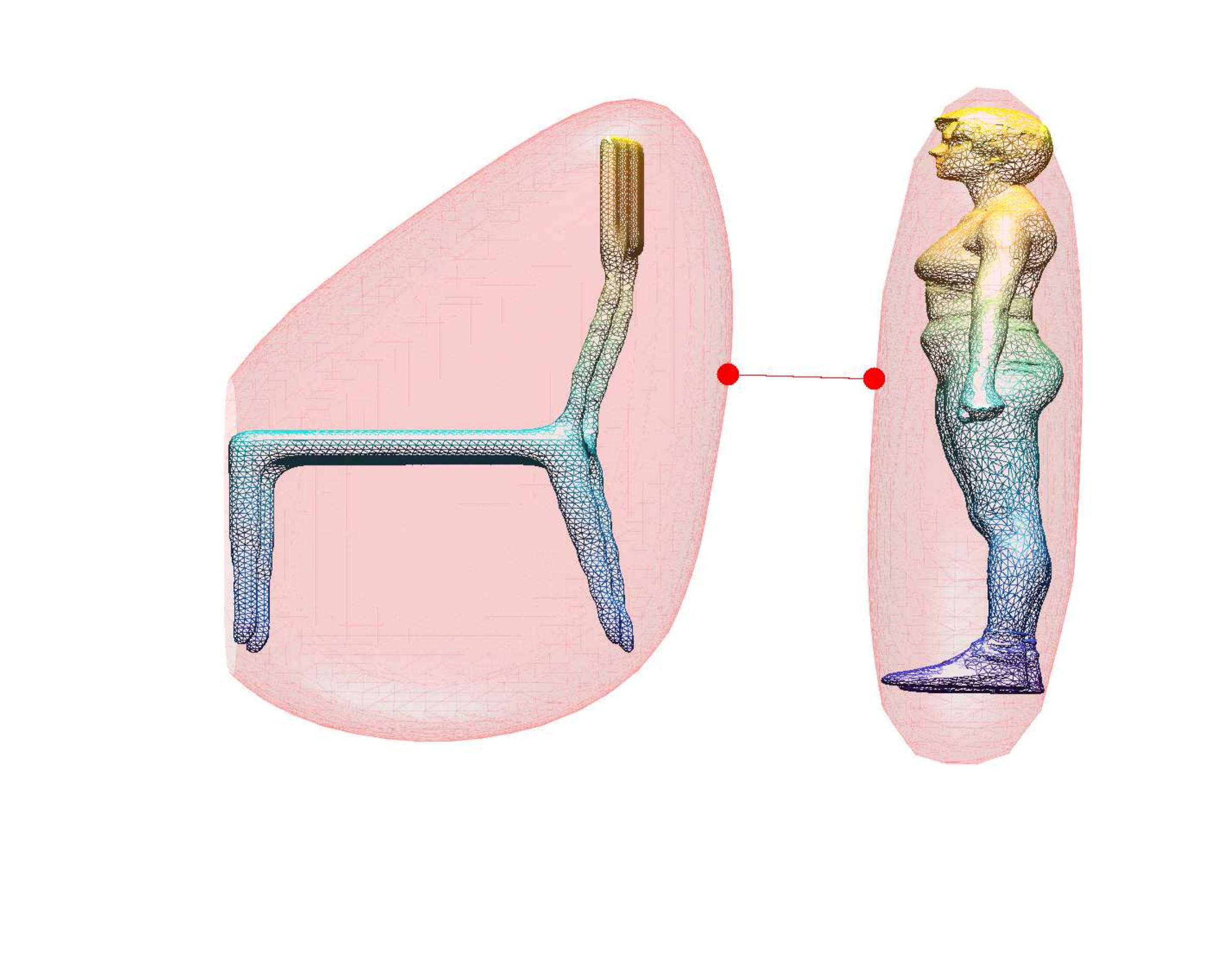}
		\caption{Minimum distance between two convex sublevel sets of degree-6 polynomials}
		\label{fig:ex.distance}
	\end{figure}

	Using MATLAB's fmincon active-set solver, the time required to compute the distance between two sos-convex bodies ranges from around 80 milliseconds to 340 milliseconds seconds as the degree is increased from $2$ to $8$; see Table~\ref{tab:distance.times}. We believe that the execution time can be improved by an order of magnitude with more efficient polynomial representations, warm starts for repeated queries, and reduced convergence tolerance for lower-precision results.}
\begin{table}[H]
	\begin{center}
		\begin{tabular}{ccccc}
			\hline
			degree & 2 & 4 & 6 & 8\\ 
			time (secs) & 0.08 & 0.083 & 0.13 & 0.34 \\
			\hline
		\end{tabular}
		\caption{Euclidean distance query times for sos-convex sets.}
		\label{tab:distance.times}
	\end{center}
\end{table}

\subsection{Penetration measures for overlapping bodies}
\begin{figure*}[t] 
	\includegraphics[height=4cm,width=0.24\linewidth]{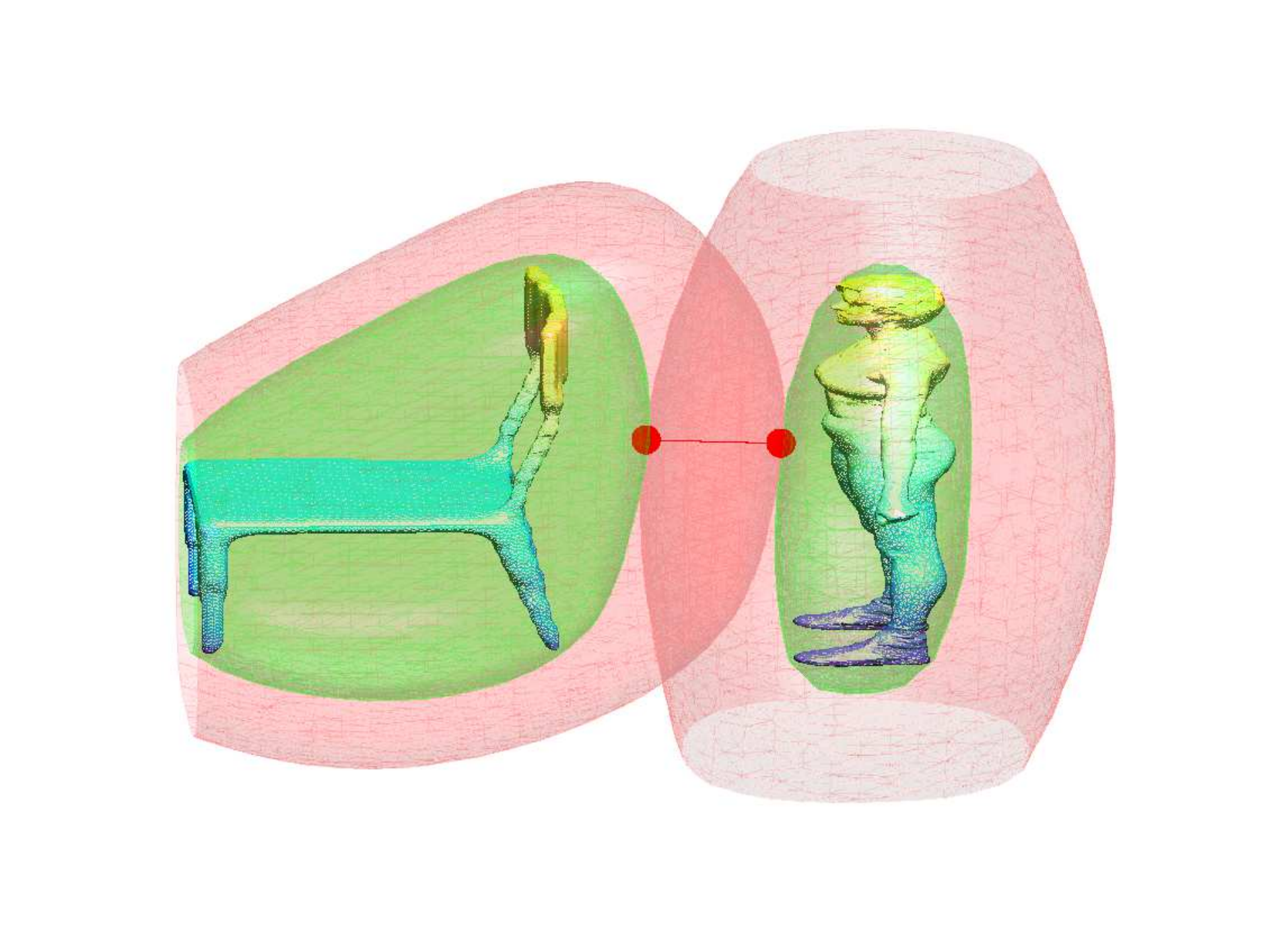}
	\includegraphics[height=4cm,width=0.24\linewidth]{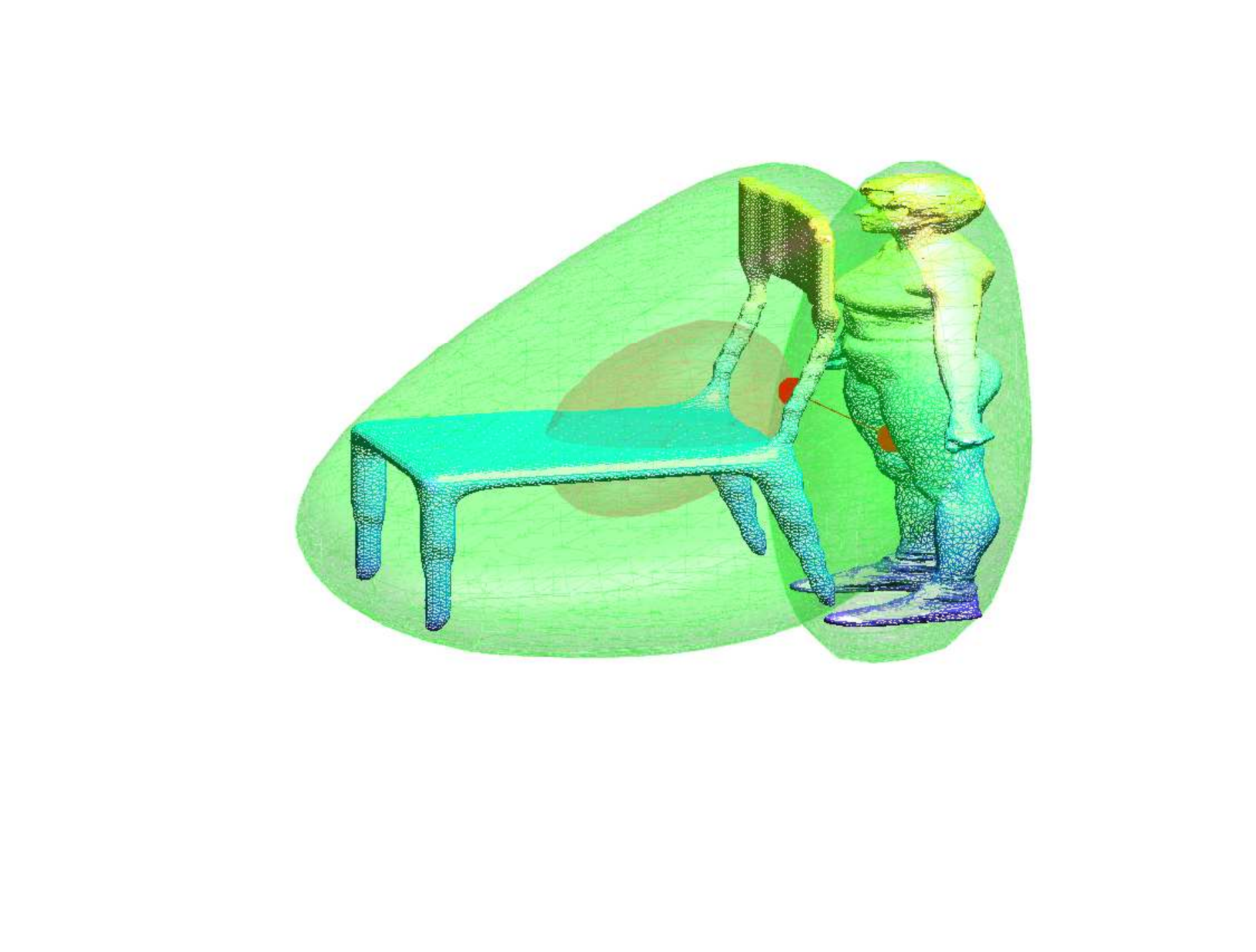}
	\includegraphics[height=4cm, width=0.24\linewidth]{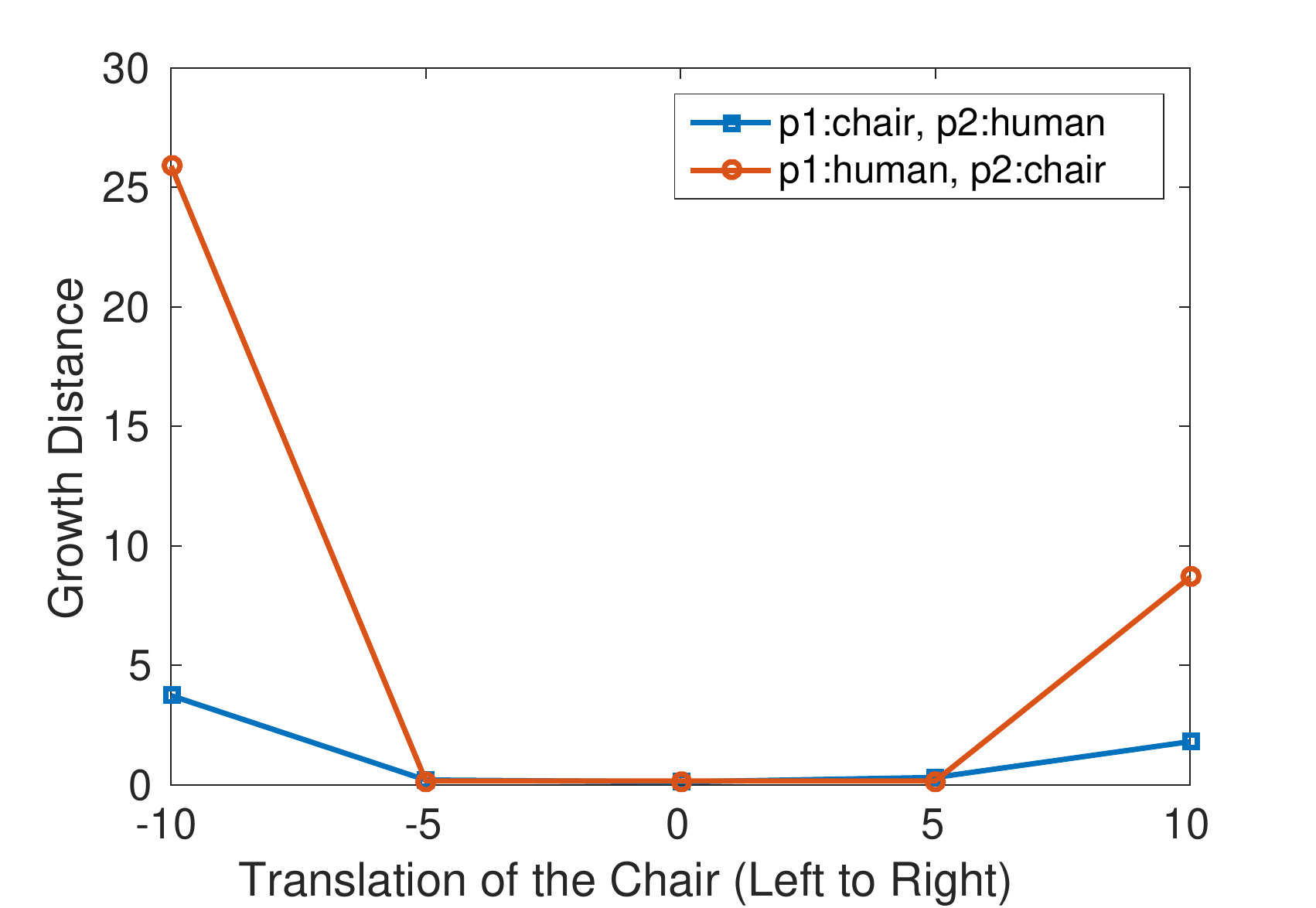}
	\includegraphics[height=4cm, width=0.24\linewidth]{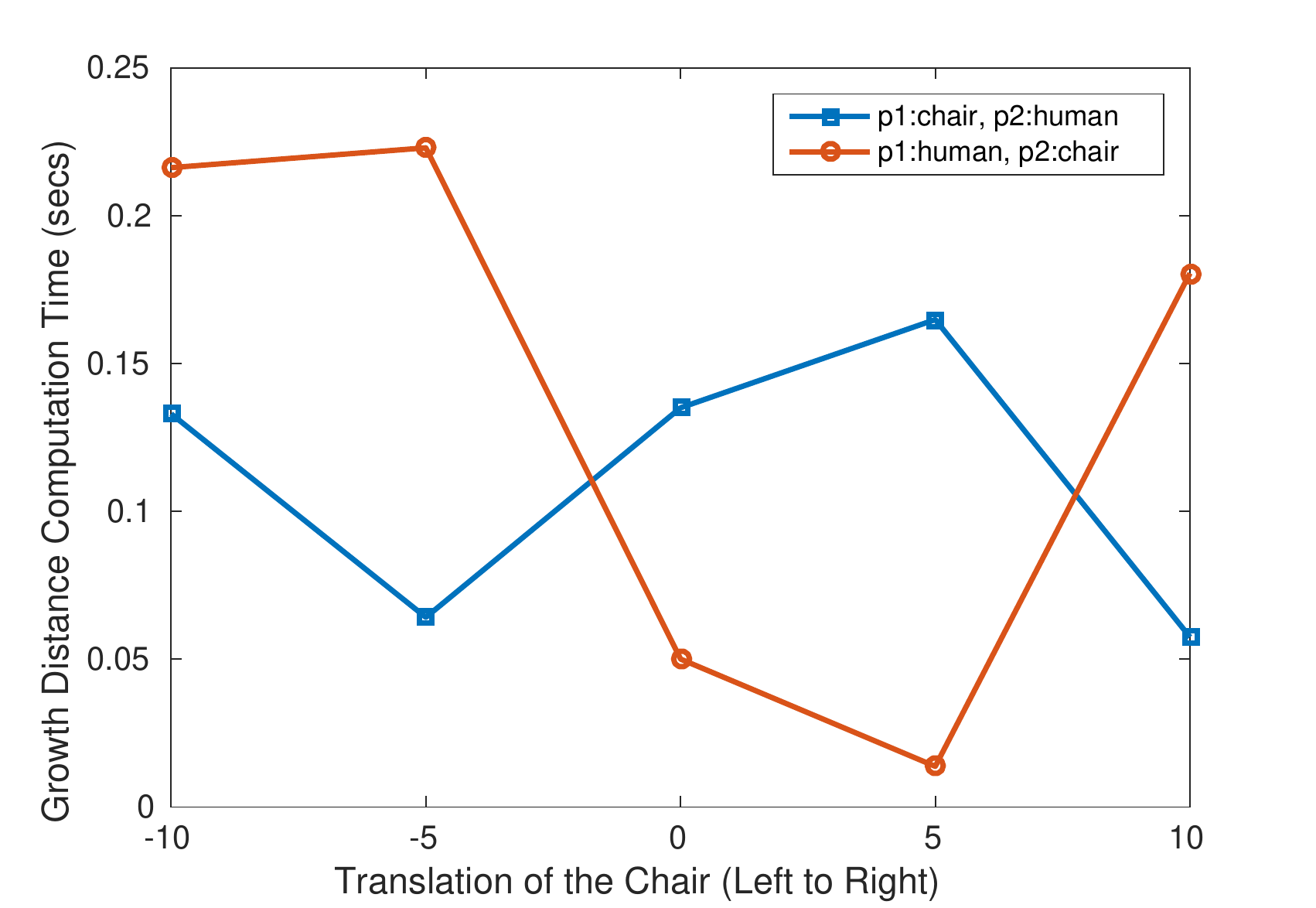}
	\caption{Growth distances for separated (left) or overlapping (second-left) sos-convex bodies; growth distance as a function of the position of the chair (second-right); time taken to solve (\ref{eq:overlap}) with warm-start (right)}
	\label{fig:penetration}
\end{figure*}
As another application of sos-convex polynomial optimization problems, we discuss a problem relevant to collision avoidance. Here, we assume that our two bodies $\mathcal{S}_1$, $\mathcal{S}_2$ are of the form $\mathcal{S}_1\mathrel{\mathop{:}}=\{x~|~ p_1(x)\leq 1\}$ and $\mathcal{S}_2\mathrel{\mathop{:}}=\{x~|~ p_2(x)\leq 1\},$ where $p_1,p_2$ are sos-convex. As shown in Figure \ref{fig:intro_pic} (right), by varying the sublevel value, we can grow or shrink the sos representation of an object. The following convex optimization problem, with optimal value denoted by $d(p_1||p_2)$, provides a measure of separation or penetration between the two bodies:
\begin{align}
&d(p_1||p_2) = \min p_1(x) \nonumber \\
&\text{s.t. } p_2(x) \leq 1. \label{eq:overlap}
\end{align} Note that the measure is asymmetric, i.e., $d(p_1||p_2)\neq d(p_2||p_1)$. It is clear that $$p_2(x) \leq 1 \Rightarrow p_1(x) \geq d(p_1||p_2).$$ In other words, the sets $\{x~|~ p_2(x) \leq 1\}$ and $\{x~|~ p_1(x) \leq d(p_1||p_2)\}$ do not overlap. As a consequence, the optimal value of (\ref{eq:overlap}) gives us a measure of how much we need to shrink the level set defined by $p_1$ to eventually move out of contact of the set $\mathcal{S}_2$ assuming that the ``seed point", i.e., the minimum of $p_1$, is outside ${\cal S}_2$. It is clear that,
\setlength{\parindent}{0em}
\begin{itemize}[nosep,leftmargin=1em,labelwidth=*,align=left]
	\item if $d(p_1||p_2) > 1$, the bounding volumes are separated.
	\item if $d(p_1||p_2) = 1$, the bounding volumes touch.
	\item if $d(p_1||p_2) < 1$, the bounding volumes overlap.
\end{itemize}

These measures are closely related to the notion of growth models and growth distances~\cite{GrowthDistance}. Note that similarly to what is described for the sos-convex case in Section \ref{subsec:eucl.dist},
the optimal solution $d(p_1||p_2)$ to (\ref{eq:overlap}) can be computed exactly using semidefinite programming, or using a generic convex optimizer. The two leftmost subfigures of Figure~\ref{fig:penetration} show a chair and a human bounded by 1-sublevel sets of degree 6 sos-convex polynomials (in green). In both cases, we compute $d(p_1||p_2)$ and $d(p_2||p_1)$ and plot the corresponding minimizers. In the first subfigure, the level set of the chair needs to grow in order to touch the human and vice-versa, certifying separation. In the second subfigure, we translate the chair across the volume occupied by the human so that they overlap. In this case, the level sets need to contract. In the third subfigure, we plot the optimal value of the problem in (\ref{eq:overlap}) as the chair is translated from left to right, showing how the growth distances dip upon penetration and rise upon separation. The final subfigure shows the time taken to solve (\ref{eq:overlap}) when warm started from the previous solution. The time taken is of the order of 150 milliseconds without warm starts to 10 milliseconds with warm starts.

\subsection{Separation and penetration under rigid body motion}
Suppose $\{x ~|~ p({x})\leq 1\}$ is a minimum-volume sos-convex body enclosing a rigid 3D object. If the object is rotated by ${R}\in SO(3)$ and translated by ${t}\in \reals^3$, then the polynomial $p'({x}) = p({R}^T {x} - {R}^T {t})$ encloses the transformed object. This is because, if $p({x}) \leq 1$, then $p'({R}{x} + {t}) \leq 1$. For continuous motion, the optimization for Euclidean distance or sublevel-based separation/penetration distances can be warm started from the previous solution. The computation of the gradient of these measures 
using parametric convex optimization, and exploring the potential of this idea for motion planning is left for future work.

\section{Containment of polynomial sublevel sets} \label{sec:cont.poly.sub}

In this section, we show how the sum of squares machinery can be used in a straightforward manner to contain polynomial sublevel sets (as opposed to point clouds) with a convex polynomial level set. More specifically, we are interested in the following problem: Given a basic semialgebraic set 

\begin{equation}\label{eq:basic.semialgebraic.set.S}
\mathcal{S}\mathrel{\mathop:}=\{x\in\mathbb{R}^n| \ g_1(x)\leq 1,\ldots, g_m(x)\leq 1   \},
\end{equation}
find a convex polynomial $p$ of degree $2d$ such that 
\begin{equation}\label{eq:S.in.sublevelset.p}
\mathcal{S}\subseteq \{x\in\mathbb{R}^n|\ p(x)\leq 1\}.
\end{equation} 
Moreover, we typically want the unit sublevel set of $p$ to have small volume. Note that if we could address this question, then we could also handle a scenario where the unit sublevel set of $p$ is required to contain the union of several basic semialgebraic sets (simply by containing each set separately). For the 3D geometric problems under our consideration, we have two applications of this task in mind:

\begin{itemize}
	\item {\bf Convexification:} In some scenarios, one may have a nonconvex outer approximation of an obstacle (e.g., obtained by the computationally inexpensive inverse moment approach of Lasserre and Pauwels as described in Section \ref{subsec:nonconvex}) and be interested in containing it with a convex set. This would e.g. make the problem of computing distances among obstacles more tractable; cf. Section \ref{sec:distance}. 
	
	\item {\bf Grouping multiple obstacles:} For various navigational tasks involving autonomous agents, one may want to have a mapping of the obstacles in the environment in varying levels of resolution. A relevant problem here is therefore to group obstacles: this would lead to the problem of containing several polynomial sublevel sets with one.
\end{itemize}

In order to solve the problem laid out above, we propose the following sos program:
\begin{align}
&\min_{p \in \mathbb{R}_{2d}[x],\tau_i \in \mathbb{R}_{2\hat{d}}[x],P \in S^{N \times N}} -\log \det(P) \nonumber\\
\text{s.t. } &p(x)=z(x)^TPz(x),  P\succeq 0, \nonumber\\
&p(x) \quad \text{sos-convex,} \label{eq:p.sosconvex}\\
& 1-p(x)-\sum_{i=1}^m \tau_i(x)(1-g_i(x)) \quad \text{sos,} \label{eq:s.procedure}\\ 
& \tau_i(x) \quad \text{sos,} \label{eq:tau.sos}\quad i=1,\ldots,m.
\end{align}


It is straightforward to see that constraints (\ref{eq:s.procedure}) and (\ref{eq:tau.sos}) imply the required set containment criterion in (\ref{eq:S.in.sublevelset.p}). As usual, the constraint in (\ref{eq:p.sosconvex}) ensures convexity of the unit sublevel set of $p$. The objective function attempts to minimize the volume of this set. A natural choice for the degree $2\hat{d}$ of the polynomials $\tau_i$ is $2\hat{d}=2d-\min_i deg(g_i)$, though better results can be obtained by increasing this parameter.

{\ghh An analoguous problem is discussed in recent work by Dabbene, Henrion, and Lagoa~\cite{Henrion1, Henrion}. In the paper, the authors want to find a polynomial $p$ of degree $d$ whose 1-superlevel set $\{x~|~ p(x)\geq 1\}$ contains a semialgebraic set $\mathcal{S}$ and has minimum volume. Assuming that one is given a set $B$ containing $\mathcal{S}$ and over which the integrals of polynomials can be efficiently computed, their method involves searching for a polynomial $p$ of degree $d$ which minimizes $\int_B p(x)dx$ while respecting the constraints $p(x)\geq 1$ on $\mathcal{S}$ and $p(x) \geq 0$ on $B$. Note that the objective is linear in the coefficients of $p$ and that these last two nonnegativity conditions can be made computationally tractable by using the sum of squares relaxation. The advantage of such a formulation lies in the fact that when the degree of the polynomial $p$ increases, the objective value of the problem converges to the true volume of the set $\mathcal{S}$.}
	

\textbf{Example.} In Figure~\ref{fig:mult.cont}, we have drawn in black three random ellipsoids and a degree-4 convex polynomial sublevel set (in yellow) containing the ellipsoids. This degree-4 polynomial was the output of the optimization problem described above where the sos multipliers $\tau_i(x)$ were chosen to have degree $2$.
\begin{figure}[ht]
	\begin{center}
		\includegraphics[scale=0.4]{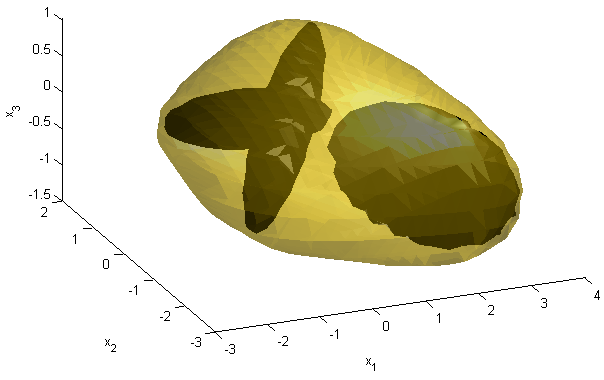}
		\caption{Containment of 3 ellipsoids using a sublevel set of a convex degree-4 polynomial}
		\label{fig:mult.cont}
	\end{center}
\end{figure}

We end by noting that the formulation proposed here is backed up theoretically by the following converse result.

\begin{theorem}
	Suppose the set $\mathcal{S}$ in (\ref{eq:basic.semialgebraic.set.S}) is Archimedean
	and that $\mathcal{S}\subset \{x\in\mathbb{R}^n|\ p(x)\leq 1\}.$ Then there exists an integer $\hat{d}$ and sum of squares polynomials $\tau_1,\ldots,\tau_m$ of degree at most $\hat{d}$ such that 
	\begin{equation}
	1-p(x)-\sum_{i=1}^m \tau_i(x)(1-g_i(x)) 
	\end{equation}
	is a sum of squares.
\end{theorem}

\begin{proof}
	The proof follows from a standard application of Putinar's Positivstellensatz \cite{putinar1993positive} and is omitted.
\end{proof}

\section{Future directions}\label{sec:conclusions}
Our results open multiple application areas for future work. 

\setlength{\parindent}{0em}
\begin{itemize}
\item Given the efficiency of our distance calculations (Section \ref{sec:distance}), it is natural to investigate performance in a real-time 3D motion planning setting. Critical objectives would be handling dynamic updates, robustness of bounding volume estimation with noisy point clouds, and investigating the viability of hierarchical bounding volumes.
\item Employing sos bodies for control of articulated objects (e.g., human motion) would provide a formulation for avoiding self-intersection, but also introduce the challenge of handling dynamic shape deformations.
\item Bounding volumes are ubiquitous in rendering applications, e.g., object culling and ray tracing. For sos bodies these ray-surface intersection operations can be framed as the distance calculations presented in Section \ref{sec:distance}. It would be interesting to explore how such techniques would perform when integrated within GPU-optimized rendering and game engine frameworks.
\item The recent works on ``dsos/sdsos" \cite{iSOS_journal} and ``dsos-convex/sdsos-convex"~\cite{DCP} polynomials provide alternatives to sos and sos-convex polynomials which are amenable to linear and second order cone programming instead of semidefinite programming. An exploration of the speed-ups offered by these approaches for bounding volume computations and their potential for use in real-time applications is left for future work. 
\end{itemize}

\section*{Acknowledgements} We thank Erwin Coumans, Mrinal Kalakrishnan and Vincent Vanhoucke for several technically insightful discussions and guidance.
\small{
\bibliographystyle{plainnat}
\bibliography{soscvx}
}
\end{document}

%% file: notation.tex
\newtheorem{theorem}{Theorem}[section]

\def\reals{\mathbb{R}}

\newcommand{\vv}[1] {\mathbf{#1}}